\documentclass[11pt,a4paper]{amsart}
\usepackage{amsfonts,amsmath,amssymb,amsthm,mathtools}
\usepackage[noadjust]{cite}
\usepackage[textsize=tiny]{todonotes}

\numberwithin{equation}{section}

\DeclareMathOperator{\Ran}{Ran}
\DeclareMathOperator{\Ker}{Ker}
\DeclareMathOperator{\Dom}{Dom}
\DeclareMathOperator{\spec}{\sigma}

\DeclarePairedDelimiter{\abs}{\lvert}{\rvert}
\DeclarePairedDelimiter{\norm}{\lVert}{\rVert}

\newcommand{\ess}{\mathrm{ess}}

\newcommand{\diag}{{\mathrm{diag}}}
\newcommand{\off}{{\mathrm{off}}}

\newcommand{\NN}{\mathbb{N}}
\newcommand{\RR}{\mathbb{R}}
\newcommand{\EE}{\mathsf{E}}

\newcommand{\cD}{{\mathcal D}}
\newcommand{\cH}{{\mathcal H}}
\newcommand{\cK}{{\mathcal K}}
\newcommand{\cM}{{\mathcal M}}
\newcommand{\cN}{{\mathcal N}}
\newcommand{\cU}{{\mathcal U}}
\newcommand{\cV}{{\mathcal V}}
\newcommand{\cW}{{\mathcal W}}

\newcommand{\fM}{{\mathfrak M}}

\theoremstyle{plain}
\newtheorem{theorem}{Theorem}[section]
\newtheorem{proposition}[theorem]{Proposition}
\newtheorem{lemma}[theorem]{Lemma}

\theoremstyle{definition}

\theoremstyle{remark}
\newtheorem{remark}[theorem]{Remark}

\title[Relative residual bounds for eigenvalues]{Relative residual bounds for eigenvalues in gaps of the essential spectrum}
\date{}

\author[A.\ Seelmann]{Albrecht Seelmann}
\address{A.~Seelmann,
 Technische Univer\-si\-t\"at Dortmund, Fakult\"at f\"ur Mathematik, D-44221 Dortmund, Germany}
\email{albrecht.seelmann@tu-dortmund.de}

\begin{document}

\begin{abstract}
  The relative distance between eigenvalues of the compression of a not necessarily semibounded self-adjoint operator to a closed
  subspace and some of the eigenvalues of the original operator in a gap of the essential spectrum is considered. It is shown that
  this distance depends on the maximal angles between pairs of associated subspaces. This generalises results by Drma\v{c} in
  [Linear Algebra Appl.~\textbf{244} (1996), 155--163] from matrices to not necessarily (semi)bounded operators.
\end{abstract}

\maketitle

\section{Introduction and main results}\label{sec:intro}

Let $H$ be a not necessarily semibounded self-adjoint operator in a Hilbert space $\cH$ with bounded inverse. We denote by
$\lambda_j \in (0,\infty)$ the $j$-th positive eigenvalue of $H$ below $\inf\spec_\ess(H) \cap (0,\infty)$, in increasing order
and counting multiplicities, provided that this eigenvalue exists.

Let $\cU$ be a finite dimensional subspace of $\Dom(H)$, and write $P_\cU$ for the orthogonal projection onto $\cU$.
In the Hilbert space $\cU$ we then consider the compression $M$ of $H$ to $\cU$, that is, the self-adjoint operator
\begin{equation*}
  M
  =
  P_\cU H|_\cU
  \colon
  \cU \to \cU
  ,
\end{equation*}
with eigenvalues
\begin{equation*}
  \mu_1
  \leq
  \dots
  \leq
  \mu_m
  ,\quad
  m = \dim\cU
  .
\end{equation*}

Under suitable additional assumptions on $\cU$, one expects at least some of the eigenvalues of $M$ to be close to certain
eigenvalues of $H$ in a relative sense; cf.\ \cite{Drm96}. In order to make this precise, consider the finite dimensional
subspaces
\begin{equation*}
  \cV = \Ran H|_\cU
  \quad\text{ and }\quad
  \cW = \Ran H^{-1}|_\cU
  ,
\end{equation*}
and denote by $P$ the (in general non-orthogonal) projection in $\cH$ onto $\cV$ along the orthogonal complement $\cW^\perp$ of
$\cW$; it will be established in Lemma~\ref{lem:projection} below that $P$ always exists and is given by $P = HP_\cU H^{-1}$. The
main result of this note now generalises Theorem~3 in \cite{Drm96} from matrices to the current setting of (unbounded) operators
$H$.

\begin{theorem}\label{thm:genFinite}
  Let $H$, $\lambda_j$, $\cU$, $M$, $\mu_k$, and $P$ be as above, and suppose that $\eta := \norm{ P_\cU - P } < 1$. Then:
  \begin{enumerate}
    \renewcommand{\theenumi}{\alph{enumi}}

    \item
    $M$ is invertible.

    \item
    If numbers $m_0, m_1 \in \NN$ with $m_0 \leq m_1 \leq \dim\cU$ satisfy $\mu_{m_0} > 0$ and $\mu_{m_1} < (1-\eta)d$, where
    $d := \inf(\spec_\ess(H) \cap (0,\infty)) \in (0,\infty]$, then $H$ has at least $m_1 - m_0 + 1$ positive eigenvalues below
    $d$, counting multiplicities, and there are indices $j_{m_0} < \dots < j_{m_1}$ with
    \begin{equation}\label{eq:relBound}
      \frac{\abs{\lambda_{j_k} - \mu_k}}{\lambda_{j_k}}
      \leq
      \eta
      \quad\text{ for all }\
      m_0
      \leq
      k
      \leq
      m_1
      .
    \end{equation}
  \end{enumerate}
\end{theorem}

Roughly speaking, Theorem~\ref{thm:genFinite} states that if $\eta < 1$, then \emph{small enough} positive eigenvalues of $M$ can
be matched to certain positive eigenvalues of $H$ with a suitable relative bound. Here, small enough refers to being well below a
threshold close to the bottom of the positive essential spectrum of $H$, cf.\ parts~(1) and~(2) of Remark~\ref{rem:genFinite}
below. As in \cite{Drm96,Gru06,GN12}, the proof of Theorem~\ref{thm:genFinite} relies on perturbing $H$ into its diagonal part
with respect to the decomposition $\Ran P_\cU \oplus \Ran (I-P_\cU)$, which is reduced by $\cU$ with corresponding part $M$, see
Section~\ref{sec:geomRelBounds} below. Note also that the subspace $\cU$ is invariant (and then, in fact, reducing) for $H$ if
and only if $\cV \subset \cU$. In this case, one even has $\cV = \cU = \cW$ and, therefore, $P = P_\cU$, see
Lemma~\ref{lem:invariant} below. In this respect, the norm of the difference $P_\cU - P$ can be regarded as an appropriate
measure for how far $\cU$ is off from being an invariant subspace for $H$. Also, if $H = H^{-1}$, then we have $\cV = \cW$ and,
thus, $P = P_\cV = P_\cW$.

\begin{remark}\label{rem:genFinite}
  (1)
  If $H$ has no positive essential spectrum at all, that is, if $d = \infty$, then the condition $\mu_{m_1} < (1-\eta)d$ in
  part~(b) of Theorem~\ref{thm:genFinite} is automatically satisfied and \emph{all} positive eigenvalues of $M$ can be matched to
  some positive eigenvalues of $H$, provided that $\eta < 1$.

  (2)
  It is worth to note that the bound \eqref{eq:relBound} together with $\mu_k < (1-\eta)d$ indeed entails $\lambda_{j_k} < d$. In
  this regard, it is a priori not possible to obtain in Theorem~\ref{thm:genFinite} analogous statements for eigenvalues
  $\mu_k \geq (1-\eta)d$. In fact, $H$ may not even have correspondingly many positive eigenvalues below $d$.

  (3)
  As already mentioned in \cite[Remark~2.3]{GV06}, a bound of the form \eqref{eq:relBound} also yields the relative bound
  \[
    \frac{\abs{\lambda_{j_k} - \mu_k}}{\mu_k}
    =
    \frac{\frac{\abs{\lambda_{j_k} - \mu_k}}{\lambda_{j_k}}}{1-\frac{\lambda_{j_k} - \mu_k}{\lambda_{j_k}}}
    \leq
    \frac{\eta}{1-\eta}
    \quad\text{ for all }\
    m_0 \leq k \leq m_1
    .
  \]

  (4)
  Upon replacing $H$ and $M$ by $-H$ and $-M$, respectively, one gets the analogous statement of Theorem~\ref{thm:genFinite} for
  negative eigenvalues in the gap of the essential spectrum.

  (5)
  Similar statements regarding eigenvalues in gaps of the essential spectrum not containing zero are also possible (while still
  keeping the requirement of bounded invertibility of $H$), but this then requires a stronger assumption on the norm
  $\norm{P_\cU - P}$ depending on the gap under consideration, see Remark~\ref{rem:gap} below. The latter can, of course,
  formally be avoided with a suitable spectral shift of $H$ (and $M$), but this then also affects the subspaces $\cV$ and $\cW$
  and, thus, the projection $P$.
\end{remark}

Let us now compare Theorem~\ref{thm:genFinite} to \cite[Theorem~3]{Drm96} and comment on other related results in the literature.

\begin{remark}\label{rem:Drmac}
  (1)
  If $J \colon \cK \to \cH$ is an isometry from some Hilbert space $\cK$ with range $\cU$, then the operator $M$ is unitarily
  equivalent to $J^*HJ$. In this sense, the above setting is consistent with the framework of~\cite{Drm96}.

  (2)
  It is easily seen that $P_\cU - P = (P_\cU - P_\cU^\perp)(P_\cU(I-P) + P_\cU^\perp P)$, where $P_\cU - P_\cU^\perp$ is unitary;
  cf.~the proof of Lemma~\ref{lem:reprHoffH} below. In particular, we have $\norm{P_\cU - P} = \norm{P_\cU(I-P) + P_\cU^\perp P}$.
  Taking into account parts~(1) of this remark and of Remark~\ref{rem:genFinite}, Theorem~\ref{thm:genFinite} therefore
  indeed contains \cite[Theorem~3]{Drm96} as a special case and, thus, generalises it from matrices to (possibly unbounded)
  operators $H$.

  (3)
  To the best of the author's knowledge, Theorem~\ref{thm:genFinite} is the first result of this kind applicable for gaps in the
  essential spectrum of not necessarily semibounded operators $H$. By contrast, for nonnegative operators $H$ stronger results
  have been obtained in \cite{GV06,GV07} for eigenvalues below the essential spectrum. In particular, \cite[Theorem~2.2]{GV06}
  allows to consider subspaces $\cU$ in the \emph{form domain} of $H$ and provides a stronger relative bound already in the case
  of matrices considered earlier in \cite{Drm96}, cf.~\cite[Example~10]{Drm96}.
\end{remark}

The following result gives a geometric bound on the norm of the difference $P_\cU - P$ in terms of the maximal angles between the
pairs of subspaces $(\cU,\cV)$, $(\cU,\cW)$, and $(\cV,\cW)$. In this regard, it recovers Proposition~5 in~\cite{Drm96} in the
current setting. Recall that the~\emph{maximal angle} $\theta(\cM,\cN)$ between two closed subspaces $\cM$ and $\cN$ of $\cH$ can
be defined as
\begin{equation*}
  \theta(\cM,\cN) = \arcsin( \norm{P_\cM - P_\cN} ),
\end{equation*}
see, e.g.,~\cite[Definition~2.1]{AM13}.

\begin{theorem}\label{thm:angleFinite}
  Let $\cU$, $\cV$, $\cW$, and $P$ be as in Theorem~\ref{thm:genFinite}. Then
  \begin{equation*}
    \norm{P_\cU - P}
    \leq
    \min\bigl\{ \sin\theta(\cU,\cV) , \sin\theta(\cU,\cW) \bigr\} + \tan\theta(\cV,\cW)
    .
  \end{equation*}
\end{theorem}

The rest of this note is organised as follows: Section~\ref{sec:relBounds} presents a general perturbation result that addresses
relative bounds for eigenvalues in gaps of the essential spectrum. In essence, it reproduces a result from \cite{Ves08} in an
operator framework, but is proved here in an alternative way using the variational principle from \cite{DES00,DES23}.
Section~\ref{sec:geomRelBounds} then adds a geometric component in terms of the projections $P_\cU$, $P_\cV$, $P_\cW$, and $P$
that allows to infer from the general result in Section~\ref{sec:relBounds} the core result of this note, Theorem~\ref{thm:gen}.
The latter includes Theorem~\ref{thm:genFinite} as a particular case, while allowing the subspace $\cU$ to have infinite
dimension. A likewise more general version of Theorem~\ref{thm:angleFinite}, Theorem~\ref{thm:angle}, is also proved in that
section utilizing known results on maximal angles between closed subspaces.

\section{Relative bounds for eigenvalues}\label{sec:relBounds}

In this section we prove a general residual bound for eigenvalues in gaps of the essential spectrum of self-adjoint operators,
which lays the foundation for the proof of Theorem~\ref{thm:genFinite}. The corresponding result essentially reproduces
\cite[Theorem~4.13]{Ves08} in the particular case of an operator framework; see also \cite[Theorem~3.16]{VS93} for the matrix
case.

For a self-adjoint operator $T$, we denote by $\EE_T$ the projection-valued spectral measure for $T$, and for $\gamma \in \RR$
we write $\lambda_{\gamma,j}(T) = \lambda_{j}(T|_{\Ran\EE_T((\gamma,\infty))}) \geq \gamma$, $j \in \NN$,
$j \leq \dim \Ran\EE_T((\gamma,\infty))$, for the $j$-th standard variational value of the lower semibounded part
$T|_{\Ran\EE((\gamma,\infty))}$ of $T$. It agrees with the $j$-th eigenvalue of $T|_{\Ran\EE((\gamma,\infty))}$ below its
essential spectrum, in nondecreasing order and counting multiplicities, if this eigenvalue exists, and otherwise equals the bottom
of the essential spectrum of $T|_{\Ran\EE((\gamma,\infty))}$. In fact, if $\Ran\EE((\gamma,\infty))$ is infinite dimensional,
then $\lambda_{\gamma,j}(T) \to \inf(\spec_\ess(T) \cap (\gamma,\infty)) \in [\gamma,\infty]$ as $j \to \infty$.

Let $A$ be self-adjoint, and let $V$ be symmetric with $\Dom(V) \supset \Dom(A)$. Suppose that for some constants $a \in \RR$,
$b \in [0,1)$ the operator $A_1 := a + b\abs{A}$ is nonnegative and that $\norm{ Vx } \leq \norm{ A_1 x }$ for all
$x \in \Dom(A)$. In particular, this gives $\norm{ Vx } \leq \abs{a} \norm{ x } + b \norm{ Ax }$ for all $x \in \Dom(A)$, so that
$B := A + V$ is self-adjoint on $\Dom(B) = \Dom(A)$ by the well-known Kato-Rellich theorem.
The following result is used in Section~\ref{sec:geomRelBounds} below only in the particular case where $a = 0$. However, the
more general case of $a \in \RR$ does not require much more efforts and is more in line with the mentioned guiding statement from
\cite{Ves08}.

\begin{proposition}\label{prop:eig}
  Let the interval $(\alpha , \beta)$ with $\beta - \alpha > 2a + b(\abs{\alpha} + \abs{\beta})$
  be in the resolvent set of $A$. Then:
  \begin{enumerate}
    \renewcommand{\theenumi}{\alph{enumi}}

    \item
    The interval $( \alpha + b\abs{\alpha} + a , \beta - b\abs{\beta} - a )$ belongs to the resolvent set of $B = A + V$.

    \item
    The subspace $\Ran\EE_A((\alpha,\infty))$ has finite dimension if and only if $\Ran\EE_B((\alpha+b\abs{\alpha}+a,\infty))$
    has finite dimension, and in this case $\dim\Ran\EE_A((\alpha,\infty)) = \dim\Ran\EE_B((\alpha+b\abs{\alpha}+a,\infty))$
    holds.

    \item
    We have
    \begin{equation*}
      \abs{ \lambda_{\alpha,j}(A) - \lambda_{\alpha+b\abs{\alpha}+a,j}(B) }
      \leq
      a + b\abs{\lambda_{\alpha,j}(A)}
    \end{equation*}
    for all $j \in \NN$ with $j \leq \dim\Ran\EE_A((\alpha,\infty))$.

    \item
    With $d := \inf( \spec_\ess(A) \cap (\alpha,\infty) ) \in [\beta , \infty]$ we have
    \begin{equation*}
      d - b\abs{d} - a
      \leq
      \inf\bigl( \spec_\ess(B) \cap (\alpha+b\abs{\alpha}+a,\infty) \bigr)
      \leq
      d + b\abs{d} + a
      ,
    \end{equation*}
    where the lower and upper bounds are interpreted as $\infty$ if $d = \infty$. In particular, the spectral part
    $\spec(B) \cap (\alpha+b\abs{\alpha}+a,\infty)$ is purely discrete if $\spec(A) \cap (\alpha,\infty)$ is purely discrete.

  \end{enumerate}
\end{proposition}

For the convenience of the reader, a proof of Proposition~\ref{prop:eig} is presented below. Other than the approach in
\cite{Ves08}, which was based on analyticity properties, this proof alternatively relies on the minimax principle from
\cite{DES00,DES23} for eigenvalues in gaps of the essential spectrum. The following proposition formulates a variant of this
result tailored to the current situation; cf.\ also \cite{SST20,DES}.

\begin{proposition}[{\cite[Theorem~1]{DES23}}]\label{prop:DES}
  Let $T$ be self-adjoint, and let $\Lambda$ be an orthogonal projection in the same Hilbert space such that $\Dom(T)$ is
  invariant for $\Lambda$. With $\cD_+ := \Dom(T) \cap \Ran\Lambda$ and $\cD_- := \Dom(T) \cap \Ran(I-\Lambda)$, suppose that
  \begin{equation}\label{eq:gapCondition}
    \nu
    :=
    \sup_{\substack{x_- \in \cD_-\\ \norm{x_-} = 1}} \langle x_- , Tx_- \rangle
    <
    \inf_{\substack{x_+ \in \cD_+\\ \norm{x_+} = 1}} \langle x_+ , Tx_+ \rangle
    .
  \end{equation}
  Then,
  \begin{equation}\label{eq:minimax}
    \lambda_{\nu,j}(T)
    =
    \inf_{\substack{\fM \subset \cD_+\\ \dim\fM = j}} \sup_{\substack{x \in \fM \oplus \cD_-\\ \norm{x} = 1}}
      \langle x , Tx \rangle
  \end{equation}
  for $j \in \NN$ with $j \leq \dim\Ran\Lambda$, and these describe all variational values of the lower semibounded
  part $T|_{\Ran\EE_T((\nu,\infty))}$ of $T$.
\end{proposition}

\begin{remark}\label{rem:DES}
  The inequality \eqref{eq:gapCondition} is usually called a \emph{gap condition} for $T$. In \cite{DES00,DES23}, the right-hand
  side of \eqref{eq:gapCondition} is replaced by the possibly larger term
  \begin{equation*}
    \inf_{x_+ \in \cD_+ \setminus \{0\}} \sup_{x_- \in \cD_-} \frac{\langle x_+ + x_- , T(x_+ + x_-)
      \rangle}{\norm{x_+ + x_-}^2},
  \end{equation*}
  which agrees with the right-hand side of \eqref{eq:minimax} for $j = 1$. In particular, the condition formulated by
  \eqref{eq:gapCondition} is stricter than the corresponding one in \cite{DES00,DES23}. However, it is exactly
  \eqref{eq:gapCondition} that is verified in the proof of Proposition~\ref{prop:eig} below.
\end{remark}

An implicit part of Proposition~\ref{prop:DES} is that under the hypotheses the subspace $\Ran\EE_T((\nu,\infty))$ has
finite dimension if and only if $\Ran\Lambda$ has, and, in this case, the two subspaces have the same dimension. Moreover, the
interval $(\nu,\lambda_{\nu,1}(T))$ belongs to the resolvent set of $T$ and, in particular, so does the interval $(\nu,\nu')$,
where $\nu'$ denotes the right-hind side of \eqref{eq:gapCondition}, cf.\ Remark~\ref{rem:DES}. With this is mind, we are ready
to prove Proposition~\ref{prop:eig}.

\begin{proof}[Proof of Proposition~\ref{prop:eig}]
  We follow the general strategy of the proof of \cite[Theorem~3.16]{VS93}.
  Since by hypothesis $A_1$ is self-adjoint and nonnegative and $V$ is symmetric with $\Dom(V) \supset \Dom(A_1) = \Dom(A)$ and
  $\norm{ Vx } \leq \norm{ A_1x }$ for all $x \in \Dom(A)$, it follows from L\"owner's theorem, see,
  e.g.,~\cite[Theorem~V.4.12]{Kato95}, that
  \begin{equation*}
    \abs{ \langle x , Vx \rangle }
    \leq
    \langle x , A_1 x \rangle
    \quad\text{ for all }\
    x \in \Dom(A)
    .
  \end{equation*}
  As a consequence, we have
  \begin{equation}\label{eq:monotonicity}
    A - A_1
    \leq
    B
    \leq
    A + A_1
  \end{equation}
  in the sense of quadratic forms, where $\Dom(A \pm A_1) = \Dom(A) = \Dom(B)$.

  In the notation of Proposition~\ref{prop:DES}, we take $\Lambda = \EE_A((\alpha,\infty)) = \EE_A([\beta,\infty))$ and
  \[
    \cD_+ = \Dom(A) \cap \Ran\EE_A([\beta,\infty))
    ,\quad
    \cD_- = \Dom(A) \cap \Ran\EE_A((-\infty,\alpha])
    .
  \]
  Define $f_\pm \colon \RR \to \RR$ by $f_\pm(t) = t \pm (a + b\abs{t})$, which both are continuous, bijective, and strictly
  increasing. Taking into account that $A \pm A_1 = f_\pm(A)$ by functional calculus, we then have
  \begin{equation*}
    \langle x, (A-A_1)x \rangle
    \geq
    f_-(\beta)\norm{x}^2
    \quad\text{ for }\quad
    x \in \cD_+
  \end{equation*}
  and
  \begin{equation*}
    \langle x, (A+A_1)x \rangle
    \leq
    f_+(\alpha)\norm{x}^2
    \quad\text{ for }\quad
    x \in \cD_-
    .
  \end{equation*}
  Moreover, the hypothesis on $\alpha$ and $\beta$ guarantees that $f_-(\beta) > f_+(\alpha)$. In light of
  \eqref{eq:monotonicity}, for each of the choices $T \in \{A\pm A_1,B\}$ the gap condition \eqref{eq:gapCondition} is therefore
  satisfied with
  \[
    \sup_{\substack{x_-\in\cD_-\\\norm{x_-}=1}} \langle x_- , Tx_- \rangle
    \leq
    f_+(\alpha)
    <
    f_-(\beta)
    \leq
    \inf_{\substack{x_+\in\cD_+\\\norm{x_+}=1}} \langle x_+ , Tx_+ \rangle
    ,
  \]
  so that Proposition~\ref{prop:DES} can be applied for all three choices. In particular, the interval $(f_+(\alpha),f_-(\beta))$
  belongs to the resolvent set of all three operators $A\pm A_1$ and $B$. With $T = B$, this proves parts (a) and (b) of the
  claim. Furthermore, with $\gamma = f_+(\alpha)$ we obtain from \eqref{eq:monotonicity} and
  the representation of the variational values in \eqref{eq:minimax} that
  \begin{equation}\label{eq:monotonicityMinimax}
    \lambda_{\gamma,j}(A-A_1)
    \leq
    \lambda_{\gamma,j}(B)
    \leq
    \lambda_{\gamma,j}(A+A_1)
  \end{equation}
  for all $j \in \NN$, $j \leq \dim\Ran\EE_A((\alpha,\infty))$. Here, we also have the representation
  $\lambda_{\gamma,j}(A - A_1) = \lambda_{f_-(\alpha),j}(f_-(A))$ since also the interval $(f_-(\alpha),f_-(\beta))$ belongs to
  the resolvent set of $f_-(A)$ by the spectral mapping theorem, as well as trivially
  $\lambda_{\gamma,j}(A + A_1) = \lambda_{f_+(\alpha),j}(f_+(A))$. In turn, again by the spectral mapping theorem, we have
  $\lambda_{f_\pm(\alpha),j}(f_\pm(A)) = f_\pm(\lambda_{\alpha,j}(A))$, so that we arrive at the representations
  $\lambda_{\gamma,j}(A\pm A_1) = f_\pm(\lambda_{\alpha,j}(A))$. Plugging the latter into \eqref{eq:monotonicityMinimax} gives
  \begin{equation}\label{eq:boundEig}
    \lambda_{\alpha,j}(A) - (a+b\abs{\lambda_{\alpha,j}(A)})
    \leq
    \lambda_{\gamma,j}(B)
    \leq
    \lambda_{\alpha,j}(A) + (a+b\abs{\lambda_{\alpha,j}(A)})
  \end{equation}
  for all $j \in \NN$ with $j \leq \dim\Ran\EE_A((\alpha,\infty))$, which proves part (c) of the claim.

  It remains to show part (d). If $\Ran\EE_A((\alpha,\infty))$ has finite dimension, then also $\Ran\EE_B((f_+(\alpha),\infty))$
  has finite dimension by part (b) and there is nothing to prove. So, suppose that $\Ran\EE_A((\alpha,\infty))$, and hence
  also $\Ran\EE_B((\gamma,\infty))$, is infinite dimensional. The claim of part (d) then follows by taking in \eqref{eq:boundEig}
  the limit as $j \to \infty$. This completes the proof.
 \end{proof}%

\section{Geometric residual bounds and proof of main results}\label{sec:geomRelBounds}

A large part of the considerations in this section also works under more general assumptions than the ones from
Section~\ref{sec:intro}. With this in mind, let $H$ and $\lambda_j$ be as in Section~\ref{sec:intro}, and let $\cU$ be a (not
necessarily finite dimensional) closed subspace such that $\Dom(H)$ is invariant for the orthogonal projection $P_\cU$ onto $\cU$;
this obviously includes the case where $\cU$ is just a finite dimensional subspace of $\Dom(H)$ as in Section~1. Let $M$ be the
compression of $H$ to $\cU$, that is,
\begin{equation}\label{eq:defM}
  M = P_\cU H|_\cU
  \quad\text{ with }\
  \Dom(M) = \Dom(H) \cap \cU \subset \cU
  ,
\end{equation}
as an operator in the Hilbert space $\cU$.
Finally, denote by $\cV$ and $\cW$ the closed subspaces
\begin{equation*}
  \cV = \Ran H|_\cU
  \quad\text{ and }\quad
  \cW = \overline{\Ran H^{-1}|_\cU}
  .
\end{equation*}
We begin with the following elementary, essentially well-known lemma.

\begin{lemma}\label{lem:invariant}
  \begin{enumerate}
    \renewcommand{\theenumi}{\alph{enumi}}

    \item
    $\Dom(H)$ is invariant also for $P_\cU^\perp = I - P_\cU$.

    \item
    $\Dom(H)$ splits as
    \begin{equation*}
      \Dom(H) = ( \Dom(H) \cap \cU ) \oplus ( \Dom(H) \cap \cU^\perp ).
    \end{equation*}

    \item
    $M$ is densely defined in $\cU$.

    \item
    If $\cU$ is invariant for $H$, then $\cV = \cW = \cU$.

  \end{enumerate}
\end{lemma}

\begin{proof}
  (a) is clear, and (b) follows immediately from (a) and the identity $I = P_\cU + P_\cU^\perp$.

  For part (c), let $u \in \cU$. Since $H$ is densely defined, we may choose a sequence $(x_k)$ in $\Dom(H)$ that converges to
  $u$. Taking into account that $P_\cU$ is bounded, the sequence $(u_k)$ with $u_k = P_\cU x_k \in \Dom(M)$ then converges to
  $P_\cU u = u$ in $\cU$, which proves the claim.

  Finally, for part (d), suppose that $\cU$ is invariant for $H$, that is, $\cV \subset \cU$. In view of part (c), a standard
  argument then shows that also $\cU^\perp$ is invariant for $H$. Now, let $y \in \cU$. By part (b), we may decompose
  $x := H^{-1}y \in \Dom(H)$ as $x = u + v$ with $u \in \Dom(H) \cap \cU$ and $v \in \Dom(H) \cap \cU^\perp$. Then, we have
  $Hu + Hv = Hx = y \in \cU$, which by $Hu \in \cU$ and $Hv \in \cU^\perp$ implies that $Hv = 0$, so that $v = 0$ because $H$ is
  invertible. We conclude that $H^{-1}y = u \in \cU$ and $y = H u \in \cV$. Since $y \in \cU$ was arbitrary and taking into
  account that $\cU$ is closed, the former yields $\cW \subset \cU$, and the latter implies $\cU \subset \cV$, that is,
  $\cU = \cV$.

  In order to show the remaining inclusion $\cU \subset \cW$, we observe that the invariance of $\cU$ for $H$ implies that
  $\Dom(H) \cap \cU \subset \Ran H^{-1}|_\cU \subset \cW$. In view of part~(c) and the closedness of $\cW$, this shows that
  indeed $\cU \subset \cW$, which completes the proof.
\end{proof}%

\begin{remark}
  The above reasoning for part (d) of Lemma~\ref{lem:invariant} is essentially contained, at least in part, in the proof of
  Lemma~2.1 in \cite{MSS16}; cf.\ also Remark~2.3 and Lemma~2.4 in \cite{TW14}.
\end{remark}

The next lemma proves the existence of the (not necessarily orthogonal) projection onto $\cV$ along $\cW^\perp$ by providing an
explicit representation in terms of $H$ and $P_\cU$.

\begin{lemma}\label{lem:projection}
  The operator $P = HP_\cU H^{-1}$ is the projection onto $\cV$ along $\cW^\perp$, that is, $P$ is bounded with $P^2 = P$ and
  satisfies $\Ran P = \cV$ and $\Ker P = \cW^\perp$.
\end{lemma}

\begin{proof}
  Observe that $P = HP_\cU H^{-1}$ is closed and everywhere defined, hence bounded by the closed graph theorem. It is then
  obvious that also $P^2 = P$. Finally, we have the identities $\Ran P = \Ran( H P_\cU|_{\Dom(H)} ) = \cV$ as well as
  $\Ker P = \Ker (P_\cU H^{-1}) = (\Ran( H^{-1} P_\cU ))^\perp = \cW^\perp$.
\end{proof}%

\begin{remark}
  More generally, if $L$ is a closed densely defined operator with bounded inverse such that $\Dom(L)$ is invariant for $P_\cU$,
  then $LP_\cU L^{-1}$ is the projection onto $\Ran L|_\cU$ along $\Ker(P_\cU L^{-1}) = (\Ran (L^{-*}|_\cU))^\perp$.
\end{remark}

In light of the domain splitting in part (b) of Lemma~\ref{lem:invariant}, we may define the diagonal and off-diagonal parts of
$H$ with respect to $\cU \oplus \cU^\perp$ as
\begin{equation*}
  H_\diag = P_\cU H P_\cU + P_\cU^\perp H P_\cU^\perp,\quad
  H_\off = P_\cU H P_\cU^\perp + P_\cU^\perp H P_\cU
\end{equation*}
with $\Dom(H_\diag) = \Dom(H) = \Dom(H_\off)$; cf.\ also~\cite{Gru06,GN12}. In particular, we have the operator identity
\[
  H
  =
  H_\diag + H_\off
  .
\]
Clearly, the subspace $\cU$ reduces $H_\diag$ in the sense that $H_\diag$ is the direct sum of operators defined in $\cU$ and
$\cU^\perp$, respectively, and $M$ is the part of $H_\diag$ associated to $\cU$. We now aim to apply Proposition~\ref{prop:eig}
from the previous section with $A = H$ and $V = -H_\off$, so that $A + V = H_\diag$. To this end, we make the following elementary
observation.
\begin{lemma}\label{lem:reprHoffH}
    We have
    \begin{equation}\label{eq:reprHoffH}
      H_\off H^{-1}
      =
      (P_\cU - P_\cU^\perp) (P_\cU - P)
    \end{equation}
    with $P$ as in Lemma~\ref{lem:projection}.
\end{lemma}

 \begin{proof}
  We calculate
  \begin{align*}
    H_\off H^{-1}
    &=
    P_\cU H P_\cU^\perp H^{-1} + P_\cU^\perp H P_\cU H^{-1}
      =
      P_\cU(I - P) + P_\cU^\perp P\\
    &=
    (P_\cU - P_\cU^\perp) (P_\cU - P).\qedhere
  \end{align*}
 \end{proof}%

Note that the factor $P_\cU - P_\cU^\perp$ on the right-hand side of~\eqref{eq:reprHoffH} is self-adjoint and unitary and can
therefore be ignored when it comes to estimating $H_\off H^{-1}$ in norm. With this in mind, we are now able to formulate and
prove the core result of this note. Here, the particular case where $\cU$ has finite dimension agrees with
Theorem~\ref{thm:genFinite}.

\begin{theorem}\label{thm:gen}
  Suppose that $\eta := \norm{ P_\cU - P } < 1$ with $P = HP_\cU H^{-1}$ as in Lemma~\ref{lem:projection}.
  \begin{enumerate}
    \renewcommand{\theenumi}{\alph{enumi}}

    \item
    The operator $M$ in \eqref{eq:defM} is self-adjoint and has a bounded inverse.

    \item
    With $d := \inf(\spec_\ess(H) \cap (0,\infty)) \in (0,\infty]$ we have 
    \begin{equation*}
      \inf\bigl( \spec_\ess(M) \cap (0,\infty) \bigr)
      \geq
      (1-\eta)d
      .
    \end{equation*}

    \item
    Denote by $(\mu_k)_{k \in J}$ with $J \subset \NN$ the (finite or infinite) collection of eigenvalues of $M$ in the
    interval $(0,(1-\eta)d)$, in increasing order and counting multiplicities. Then, there is a family of indices $j_k \in \NN$,
    $k \in J$, strictly increasing in $k$, such that for each $k \in J$ we have
    \begin{equation*}
      \frac{\abs{\lambda_{j_k} - \mu_k}}{\lambda_{j_k}}
      \leq
      \eta
      .
    \end{equation*}

  \end{enumerate}
\end{theorem}

\begin{proof}
  In view of Lemma~\ref{lem:reprHoffH}, we have $\norm{H_\off H^{-1}} = \norm{P_\cU-P} = \eta < 1$. In
  particular, this gives
  \begin{equation*}
    \norm{H_\off x}
    \leq
    \eta\norm{Hx}
    =
    \eta\norm{\abs{H}x}
  \end{equation*}
  for all $x \in \Dom(H)$. In the notation of Section~\ref{sec:relBounds}, we may therefore take $A = H$ and $V = -H_\off$ with
  $a = 0$ and $b = \eta \in [0,1)$. Moreover, since $H$ has a bounded inverse, there are numbers $\alpha,\beta \in \RR$ with
  $\alpha < 0 < \beta$ such that the interval $(\alpha,\beta)$ belongs to the resolvent set of $H$; in particular, we have
  \[
    d
    =
    \inf(\spec_\ess(H) \cap (0,\infty))
    =
    \inf(\spec_\ess(H) \cap (\alpha,\infty))
    \geq
    \beta
    .
  \]
  We observe that $2a + b(\abs{\alpha}+\abs{\beta}) = \eta(\beta-\alpha) < \beta - \alpha$, so that the hypotheses of
  Proposition~\ref{prop:eig} are satisfied. We conclude that $H_\diag = H - H_\off$ is self-adjoint and that the interval
  $((1-\eta)\alpha,(1-\eta)\beta)$ belongs to its resolvent set; in particular, $H_\diag$ has a bounded inverse. Moreover,
  part~(d) of Proposition~\ref{prop:eig} gives
  \[
    \inf(\spec_\ess(H_\diag) \cap ((1-\eta)\alpha,\infty))
    \geq
    (1-\eta)d
    .
  \]
  Since $\cU$ reduces $H_\diag$ and $M$ is the part of $H_\diag$ associated to $\cU$, this proves (a) and (b).

  Taking into account that each $\lambda_{\alpha,j}(H)$ is positive, it follows from part~(c) of Proposition~\ref{prop:eig} that
  \begin{equation}\label{eq:eigHHdiag}
    \frac{\abs{\lambda_{\alpha,j}(H) - \lambda_{(1-\eta)\alpha,j}(H_\diag)}}{\lambda_{\alpha,j}(H)}
    \leq
    \eta
  \end{equation}
  for all $j \in \NN$ with $j \leq \dim\Ran\EE_H((\alpha,\infty))$. In particular, this implies that $\lambda_{\alpha,j}(H) < d$
  if $\lambda_{(1-\eta)\alpha,j}(H_\diag) < (1-\eta)d$. Now, by definition of the $\mu_k$ there are indices $j_k$ with
  $\lambda_{(1-\eta)\alpha,j_k}(H_\diag) = \mu_k \in (0,(1-\eta)d)$ for all $k \in J$. Thus, $\lambda_{\alpha,j_k}(H) < d$ is the
  $j_k$-th positive eigenvalue of $H$ below $d$, that is, $\lambda_{\alpha,j_k}(H) = \lambda_{j_k}$. Together with
  \eqref{eq:eigHHdiag}, this shows part~(c) and, hence, completes the proof of the theorem.
\end{proof}%

Let us collect some useful observations regarding part~(b) of Theorem~\ref{thm:gen}.
\begin{remark}\label{rem:main}
  (1)
  The proof of Theorem~\ref{thm:gen} gives
  \begin{equation*}
    (1-\eta)d
    \leq
    \inf(\spec_\ess(H_\diag) \cap ((1-\eta)\alpha,\infty))
    \leq
    \inf(\spec_\ess(M) \cap ((1-\eta)\alpha,\infty))
    ,
  \end{equation*}
  and either inequality may a priori be strict. Thus, eigenvalues of $M$ that are larger than (or equal to) $(1-\eta)d$ are not
  necessarily accessible via the variational values $\lambda_{(1-\eta)\alpha,j}(H_\diag)$, and even for those that are accessible,
  we can no longer guarantee that the corresponding variational values $\lambda_{\alpha,j}(H)$ for $H$ are smaller than $d$. The
  latter may therefore not correspond to eigenvalues of $H$.

  (2)
  If $\cU$ has finite dimension, then
  \[
    \inf(\spec_\ess(H_\diag) \cap ((1-\eta)\alpha,\infty))
    =
    \inf(\spec_\ess(H) \cap (\alpha,\infty))
    =
    d
    .
  \]
  Indeed, in this case $\Ran(P_\cU-P)$ has finite dimension and, consequently, in view of Lemma~\ref{lem:reprHoffH},
  $H_\diag^{-1} - H^{-1} = H_\diag^{-1}H_\off H^{-1}$ is compact. Hence, $\spec_\ess(H_\diag) = \spec_\ess(H)$, see, e.g.,
  \cite[Theorem~IV.5.35]{Kato95}.
  
  (3)
  Although $\inf(\spec_\ess(H_\diag) \cap ((1-\eta)\alpha,\infty)) \leq (1+\eta)d$ by part~(d) of Proposition~\ref{prop:eig}, the
  term $\inf(\spec_\ess(M) \cap ((1-\eta)\alpha,\infty))$ might a priori be a lot larger, for instance if $H_\diag$ has positive
  essential spectrum but $M$ does not. In view of part~(2) of this remark, this is the case, in particular, if $H$ has positive
  essential spectrum and $\cU$ has finite dimension.
\end{remark}

The following remark addresses an extension of Theorem~\ref{thm:gen} to gaps of the essential spectrum of $H$ that do not contain
zero.

\begin{remark}\label{rem:gap}
  The general form of Proposition~\ref{prop:eig} allows to obtain also similar statements as in Theorem~\ref{thm:gen} for
  eigenvalues in gaps of the essential spectrum not containing zero. More precisely, instead of the interval $(\alpha,\beta)$ in
  the proof of Theorem~\ref{thm:gen}, we may consider any interval $(\tilde{\alpha},\tilde{\beta})$ belonging to the resolvent
  set of $H$ such that $\eta = \norm{P_\cU - P}$ satisfies the (stronger) condition
  \begin{equation*}
    \eta
    <
    \frac{\tilde{\beta}-\tilde{\alpha}}{\abs{\tilde{\alpha}}+\abs{\tilde{\beta}}}
    .
  \end{equation*}
  The terms $(1-\eta)\alpha$, $(1-\eta)\beta$, and $(1-\eta)d$ in the proof then just have to be replaced by
  $\tilde{\alpha} + \eta\abs{\tilde{\alpha}}$, $\tilde{\beta} - \eta\abs{\tilde{\beta}}$, and $\tilde{d} - \eta\abs{\tilde{d}}$,
  respectively, where $\tilde{d}$ is given by $\tilde{d} = \inf(\spec_\ess(H) \cap (\tilde{\alpha},\infty)) \geq \tilde{\beta}$.
\end{remark}

Theorem~\ref{thm:gen} relies on the crucial condition $\norm{ P_\cU - P } < 1$, so let us now address how the norm of $P_\cU - P$
can be estimated. To this end, we may choose one of the alternative decompositions
\begin{equation}\label{eq:projDecomp}
  P_\cU - P
  =
  ( P_\cU - P_\cV ) + ( P_\cV - P )
  =
  ( P_\cU - P_\cW ) + ( P_\cW - P)
  .
\end{equation}
Here, the terms $P_\cU - P_\cV$ and $P_\cU - P_\cW$ correspond to sines of the operator angles associated to the pairs of
subspaces $(\cU,\cV)$ and $(\cU,\cW)$, respectively. More precisely,
\begin{equation}\label{eq:opAngle}
  \abs{P_\cU - P_\cV} = \sin\Theta(\cU,\cV)
  \quad\text{ and }\quad
  \abs{P_\cU - P_\cW} = \sin\Theta(\cU,\cW),
\end{equation}
where $\Theta(\cdot,\cdot)$ denotes the operator angle associated with the respective subspaces, see,
e.g.,~\cite[Section~2]{Seel14} and the references cited therein for a discussion. In particular, the maximal angle introduced in
Section~\ref{sec:intro} satisfies $\theta(\cdot,\cdot) = \norm{\Theta(\cdot,\cdot)}$.

In order to address the other two terms, $P_\cV - P$ and $P_\cW - P$, we make the following considerations:

Since $\Ran(I_\cH - P) = \Ker P = \cW^\perp$, we obtain from $P + (I-P) = I$ that $P_\cW P = P_\cW$. Hence, the projection $P$
can be represented with respect to the orthogonal decomposition $\cW \oplus \cW^\perp$ as the $2\times 2$ block operator matrix
\begin{equation}\label{eq:P}
  P
  =
  \begin{pmatrix}
    I_\cW & 0\\
    X & 0
  \end{pmatrix}
\end{equation}
with $X := P_\cW^\perp P|_\cW$, interpreted as an operator from $\cW$ to $\cW^\perp$. In particular, $\cV = \Ran P$ admits the
graph subspace representation
\begin{equation}\label{eq:V}
  \cV = \{ f \oplus Xf \colon f \in \cW \}.
\end{equation}
Recall from~\cite[Corollary~3.4 and Remark~3.6]{KMM03:181} that consequently we have $\norm{ P_\cW - P_\cV} < 1$ and that $X$
corresponds to the tangent of the operator angle associated to the subspaces $\cW$ and $\cV$, more precisely
\begin{equation}\label{eq:reprX}
  \begin{pmatrix} \abs{X} & 0\\ 0 & \abs{X^*} \end{pmatrix}
  =
  \tan\Theta(\cW,\cV)
  .
\end{equation}
Moreover, we have
\begin{equation}\label{eq:PV}
  P_\cV = U P_\cW U^*,
\end{equation}
where $U$ is the unitary operator given by the $2\times 2$ block operator matrix
\begin{equation}\label{eq:U}
  U
  =
  \begin{pmatrix}
    (I_\cW + X^*X)^{-1/2} & -X^*(I_{\cW^\perp} + XX^*)^{-1/2}\\
    X(I_\cW + X^*X)^{-1/2} & (I_{\cW^\perp} + XX^*)^{-1/2}
  \end{pmatrix}
  .
\end{equation}
A broader discussion on the operator angle and graph subspace representations can be found, for instance, in
\cite[Sections~1.3 and~1.5]{SeelDiss} and the references cited therein.

\begin{remark}
  The inequality $\norm{ P_\cW - P_\cV } < 1$ can alternatively also be verified as follows: Since the projection $P$ onto $\cV$
  along $\cW^\perp$ exists by Lemma~\ref{lem:projection}, Proposition~1.6 in \cite{BS10} yields that
  $\norm{ P_\cV P_\cW^\perp } < 1$. Taking into account that $P^*$ is the projection onto $\cW$ along $\cV^\perp$, we obtain in
  the same way that $\norm{ P_\cV^\perp P_\cW } = \norm{ P_\cW P_\cV^\perp } < 1$. Using
  $\norm{ P_\cW - P_\cV } = \max\{ \norm{P_\cV P_\cW^\perp} , \norm{P_\cV^\perp P_\cW} \}$, see, e.g., \cite[Section~34]{AG93},
  this gives $\norm{ P_\cW - P_\cV } < 1$.
\end{remark}

\begin{lemma}\label{lem:annular}
  With $X = P_\cW^\perp P|_\cW \colon \cW \to \cW^\perp$ and $U$ as in \eqref{eq:U} we have
  \begin{equation*}
    P_\cW - P = \begin{pmatrix} 0 & 0\\ -X & 0 \end{pmatrix}
  \end{equation*}
  and
  \begin{equation*}
    P_\cV - P = U \begin{pmatrix} 0 & X^*\\ 0 & 0 \end{pmatrix} U^*.
  \end{equation*}
\end{lemma}

\begin{proof}
  The representation for $P_\cW - P$ follows directly from~\eqref{eq:P}. Moreover,
  using the identity $X^*(I_{\cW^\perp} + XX^*)^{-1/2} = (I_\cW + X^*X)^{-1/2}X^*$, the representation for $P_\cV - P$ is verified
  from~\eqref{eq:P},~\eqref{eq:PV}, and~\eqref{eq:U} by plain multiplication of $2\times 2$ block operator matrices.
\end{proof}%

We now arrive at the following result, the particular case of which where $\cU$ has finite dimension agrees with
Theorem~\ref{thm:angleFinite}.

\begin{theorem}\label{thm:angle}
  We have
  \begin{equation*}
    \norm{P_\cU - P}
    \leq
    \min\bigl\{ \sin\theta(\cU,\cV) , \sin\theta(\cU,\cW) \bigr\} + \tan\theta(\cV,\cW)
    .
  \end{equation*}
\end{theorem}

\begin{proof}
  From Lemma~\ref{lem:annular} and \eqref{eq:reprX} we obtain that
  \[
    \norm{ P_\cW - P }
    =
    \norm{ P_\cV - P }
    =
    \norm{ X }
    =
    \tan \theta(\cV , \cW)
    ,
  \]
  where for the last equality we used that $\norm{ \tan\Theta(\cW , \cV) } = \tan \theta(\cV , \cW)$.
  Combining the latter with \eqref{eq:projDecomp} and \eqref{eq:opAngle} gives
  \begin{align*}
    \norm{P_\cU - P}
    &\leq
    \min\bigl\{ \norm{P_\cU-P_\cV} , \norm{P_\cU-P_\cW} \bigr\} + \norm{X}\\
    &=
    \min\bigl\{ \sin\theta(\cU,\cV) , \sin\theta(\cU,\cW) \bigr\} + \tan\theta(\cV,\cW)
    ,
  \end{align*}
  which proves the claim.
\end{proof}%

\section*{Acknowledgements}
The author is grateful to Zlatko Drma\v{c} and Ivan Veseli\'c for suggesting this research direction. He also thanks Kre\v{s}imir
Veseli\'c for a helpful communication.

%%%%%%%%%%%%%%%%%%%%%%%%%%%%%%%%%%%%%%%%%%%%%%%%%%%%%%%%%%%%%%%%%%%%%%%%%%%%%%%%%%%%%%%%%%%%%%%%%%%%%%%%%%%%%%%%%%%%%%%%%%%%%%%%%%%
%%%%%%%%%%%%%%%%%%%%%%%%%%%%%%%%%%%%%%%%%%%%%%%%%%%%%%%%%%%%%%%%%%%%%%%%%%%%%%%%%%%%%%%%%%%%%%%%%%%%%%%%%%%%%%%%%%%%%%%%%%%%%%%%%%%
%%% Bibliography
%%%%%%%%%%%%%%%%%%%%%%%%%%%%%%%%%%%%%%%%%%%%%%%%%%%%%%%%%%%%%%%%%%%%%%%%%%%%%%%%%%%%%%%%%%%%%%%%%%%%%%%%%%%%%%%%%%%%%%%%%%%%%%%%%%%
%%%%%%%%%%%%%%%%%%%%%%%%%%%%%%%%%%%%%%%%%%%%%%%%%%%%%%%%%%%%%%%%%%%%%%%%%%%%%%%%%%%%%%%%%%%%%%%%%%%%%%%%%%%%%%%%%%%%%%%%%%%%%%%%%%%

\end{document}